\documentclass[noams]{compositio}

\usepackage{bbm}
\usepackage{eufrak}

\usepackage[all]{xy}

\renewcommand{\contentsname}{Contents\\{\footnotesize\normalfont(A table
of contents should normally not be included)}}
\newcommand\NN{{\mathcal  N}} 
 
\newcommand\OO{{\mathcal  O}} 
\newcommand\II{{\mathcal  I}} 

\newcommand\CC{{\mathfrak   S}}

\newcommand\Cs{{\mathbb C}}

\newcommand\End{\mathop {\fam 0 End}\nolimits}
\newcommand\Hom{\mathop {\fam 0 Hom}\nolimits}
\newcommand\At{\mathop {\fam 0 At}\nolimits}
\newcommand\tr{\mathop {\fam 0 tr}\nolimits}
\newcommand\id{\mathop {\fam 0 id}\nolimits} 

\newcommand\rank{\mathop {\fam 0 rank}\nolimits}

\newtheorem{thm}{Theorem} 
\theoremstyle{remark}

\newtheorem{prop}{Proposition}

\begin{document}

\title{On the obstruction to extending a  vector bundle from a submanifold}
\author{Alexey V.  Gavrilov}
\email{gavrilov@phys.nsu.ru}
\address{Department of Physics, Novosibirsk State University, 2 Pirogov Street, Novosibirsk, 630090, Russia}
\classification{32L05, 32L10,  32C25.}
\keywords{vector bundle, Atiyah class, submanifold}

\begin{abstract}

For  the first obstruction to extending a holomorphic vector bundle from a submanifold of a complex manifold found by Griffiths we give an explicit formula in terms of the Atiyah class of the bundle. 
\end{abstract}

\maketitle


\section{Introduction}
\label{sec:introduction}

Let $Y$  be a complex manifold,   $X\subset Y$ a compact submanifold, and $\imath:X\to Y$ the embedding. 
For  a  holomorphic vector bundle  $E\to X$, is  there a holomorphic  bundle  $F\to Y$
such that $E=\imath^*F$?    This  problem was considered by Griffiths who found a sequence of obstructions  to such an extension \cite[Proposition 1.1]{Griffiths}  denoted  by $\omega(\alpha_\mu),\,\,\mu\ge 0$.   However, a  proof of  Proposition 1.1   is  omitted  and  there is no actual construction of this invariants except for  line bundles.  In this note we give an explicit formula for the first obstruction. 

Following \cite{Griffiths}  we consider the analytic space  $X_1=(X, \OO_Y/\II^2)$ where   $\II\subset \OO_Y$ is  the ideal sheaf of $X$,
and interpret the  embedding of $X$ into $Y$ as   a composition  of two morphisms $X\hookrightarrow  X_1\hookrightarrow  Y$. If 
$E=\imath^*F$ then there is also  an intermediate  bundle $E_1\to X_1$  such that $E=\imath^*E_1$ which is the natural  first step in  extending $E$ over $Y$. 
(We use  the same symbol $\imath$  for  both embeddings  $X\to X_1$ and $X\to Y$. This should do 
no harm as its meaning must  be obvious from the context.) The  obstruction to such a bundle  can be described  as follows. There is a certain cohomology class  $\kappa_{X\slash Y}\in H^1(X, \Hom(\NN_{X\slash Y}, T_X))$ which we call a Kodaira-Spencer invariant (for reasons explained below). Denote 
$$ \At(E)\cdot \kappa_{X\slash Y}  =\tr(\At(E)\cup \kappa_{X\slash Y}, T_X)\in H^2(X,\NN_{X\slash Y}^*\otimes\End(E))$$
where $\At(E)\in H^1(X, \Omega_X^1\otimes\End(E))$ is the  Atiyah class of the bundle.

\begin{thm}  For a  holomorphic vector bundle $E\to X$ there is a bundle  $E_1\to X_1$ such that $E=\imath^*E_1$
iff $\At(E)\cdot\kappa_{X\slash Y}=0$. 
\end{thm}  

 So, $\omega(\alpha_0)$  is actually a combination of  the Atiyah class (which has nothing to do with the embedding) and the  Kodaira-Spencer  invariant  (which has nothing to do with the bundle). Unfortunately,  higher order obstructions  probably cannot be separated in  this  manner into the  embedding  part  and the bundle part. 

\section{Kodaira-Spencer invariant }
\label{sec:invariant} 

For a sufficiently fine cover $\{U_i\}$ of $X$ we may chose holomorphic sections 
$\Psi_i\in\Gamma(U_i, \Hom(\NN_{X\slash Y}, T_Y))$
such that for the natural   map $r:  \Hom(\NN_{X\slash Y}, T_Y)\to  \Hom(\NN_{X\slash Y}, \NN_{X\slash Y})$ 
we have $r\circ \Psi_i=\mathbbm{1}$. The Kodaira-Spencer invariant   $\kappa_{X\slash Y}\in H^1(X, \Hom(\NN_{X\slash Y}, T_X))$
can be defined as  the cohomology class of the cocycle 
$$\Psi_i-\Psi_j\in\Gamma(U_i\cap U_j, \Hom(\NN_{X\slash Y}, T_X)).$$
It   does not depend on the choice of $\Psi_i$   so   this  is indeed  an  invariant of $X$ as a  submanifold. 

If  $Y\to B$ is   a deformation of $X=X_0$ then   $\NN_{X\slash Y}$ is a trivial bundle naturally isomorphic to the tangent space $T_{B, 0}$. In this case  we  may interpret the  usual  Kodaira-Spencer map  $H^0(X, \NN_{X\slash Y})\to H^1(X, T_X)$ as a cohomology class in $H^1(X, \Hom(\NN_{X\slash Y}, T_X))$, and this class  is equal to  $\kappa_{X\slash Y}$.  A generalization  of  the  Kodaira-Spencer construction to  the  case  when $\NN_{X\slash Y}$ is not necessarily trivial seems  straightforward  enough  but  the author does not know if  it  actually  appeared  in  the literature. However, this   invariant does appear    in \cite{Griffiths} in a different guise,  it is  the first  obstruction to 
a map $f: Y\to X$ such that $f\circ \imath=\id$  \cite[Proposition 1.6]{Griffiths}.  
In hindsight, it is not very surprizing that this obstruction is related to extensions of bundles: 
if such a map exists then one obviously can  take $F=f^*E$.

\section{Atiyah class}
\label{sec:Atiyah} 

For sections of the bundle $\End(E)$ we use a notation which seems quite convenient but  probably  uncommon. 
We assume forth that a cover $\{U_i\}$  of $X$ is fixed together with 
trivializations of $E$ over $U_i$, so we have fixed   transition functions $g_{ij}: U_i\cap U_j\to GL(r,\Cs)$. 
Let $U\subset X$ be an open set and $m_\lambda: U\cap U_\lambda\to \End(\Cs^r)$ be some holomorphic matrices
where $r=\rank E$. If matrices $g_{i\lambda}m_\lambda g_{\lambda i}$ do not depend on $\lambda$ (where they are defined) then we can interpret the collection $\{m_\lambda\}$ as a section of $\End(E)$   and write $m_\lambda\in\Gamma(U, \End(E))$. (To make it  a consistent notation  we have to treat  $\lambda$ as  a special  index used  for this purpose only.)

 In this  notation, the  Atiyah class $\At(E)$ is  represented \cite{Atiyah}\cite{Kapranov}  by the  cocycle 
$$g_{\lambda i}dg_{i\lambda}-g_{\lambda j}dg_{j\lambda}\in  \Gamma(U_i\cap U_j, \Omega_X^1\otimes \End(E)).$$
(By the way,  we cannot write $g_{\lambda i}dg_{i\lambda}\in \Gamma(U_i, \Omega_X^1\otimes \End(E))$ because this  term  alone is not a well defined section of   $\Omega_X^1\otimes \End(E)$.)  If there is a bundle $F\to Y$ such that $E=\imath^*F$ then its Atiyah class can be introduced  the same way, but  we are only interested in its  restriction  to $X$ corresponding to the cocycle 
$g_{\lambda i}dG_{i\lambda}-g_{\lambda j}dG_{j\lambda}\in  \Gamma(U_i\cap U_j, \Omega_{Y|X}^1\otimes \End(E))$ where $G_{ij}$ are   transition functions  of $F$ such that $G_{ij}=g_{ij}$ on  $U_i\cap U_j$.

\section{Vector bundles over $X_1$}
\label{sec:vector}

A section  of the structure sheaf of $X_1$ may be interpreted as a pair $(f,\beta)\in \OO_{X_1}(U)$  
where $f\in\OO_X(U)$ and $\beta\in \Omega_{Y|X}^1(U)$ is a form such that  $\imath^*\beta=df$. 
Consequently, transition functions of a vector bundle  $E_1\to X_1$ are pairs $(g_{ij}, \Delta_{ij})$ 
where  $g_{ij}$ are the transition functions of $\imath^* E_1$ and $\Delta_{ij}\in  \Gamma(U_i\cap U_j, \Omega_{Y|X}^1\otimes \End(\Cs^r))$ are forms such that 
$\imath^*\Delta_{ij}=dg_{ij}$ and $\Delta_{ij}g_{ji}+g_{ij}\Delta_{jk}g_{ki}+g_{ik}\Delta_{ki}=0$. 
We say that two extensions  $E_1,E_1'\to X_1$  of the bundle $E\to X$ are equivalent  if they are isomorphic as vector
bundles over $X_1$. (Griffiths used the same definition  \cite[I.4]{Griffiths}. It  is actually  the finest  equivalence  in this context  even if   it apparently   ignores  the original bundle $E$.)
Assuming  as before  that  $g_{ij}$ are  fixed, it is easy to see that  two  forms $\Delta_{ij}$ and $\Delta'_{ij}$ correspond to equivalent bundles  if $\Delta'_{ij}-\Delta_{ij}=B_i g_{ij}-g_{ij}B_j$
where $B_i\in \Gamma(U_i, \Omega_{Y|X}^1\otimes \End(\Cs^r))$ satisfy $\imath^*B_i=0$.

The Atiyah class $\At(E_1)\in H^1(X, \Omega_{Y|X}^1\otimes \End(E))$ of such a bundle  can be defined  by the cocycle 
$A_{ij}=g_{\lambda i}\Delta_{i\lambda}-g_{\lambda j}\Delta_{j\lambda}$  
(for equivalent bundles we have  $A'_{ij}-A_{ij}=g_{\lambda i}B_ig_{i\lambda}-g_{\lambda j}B_jg_{j\lambda}$ which is, of course, a coboundary).  
Note that the equality $\At(E_1)=\At(E_1')$  gives  another  equivalence relation, in general more coarse then the natural one. 

\section{The proof}
\label{sec:proof}

\begin{prop}
An extension $E_1\to X_1$ of a holomorphic vector bundle $E\to X$ exists iff $\At(E)$ is in the image of the  map
$$\imath^*: H^1(X,  \Omega_{Y|X}^1\otimes\End(E))\to  H^1(X, \Omega_X^1\otimes\End(E)).$$
\end{prop}

\begin{proof}

If there is an extension then $\At(E)=\imath^*\At(E_1)$  so this condition is  necessary. Conversely, let 
$D_{ij}\in \Gamma(U_i\cap U_j, \Omega_{Y|X}^1\otimes \End(E))$ be a cocycle representing a cohomology class with the image
$\At(E)$. Then we have 
$$\imath^*D_{ij}=g_{\lambda i}dg_{i\lambda}-g_{\lambda j}dg_{j\lambda}+g_{\lambda i}C_ig_{i\lambda}-g_{\lambda j}C_jg_{j\lambda} $$
for some $C_i\in  \Gamma(U_i, \Omega_{X}^1\otimes \End(\Cs^r))$.  The coboundary part can be lifted to $\Omega_{Y|X}^1$ so we are free to  assume that $C_i=0$.
    Then $(g_{ij}, \Delta_{ij})$ with   $\Delta_{ij}=g_{i\lambda}D_{ij}g_{\lambda j}$ are  transition functions of a vector bundle over $X_1$  with this Atiyah class. 
\end{proof}

\begin{proof}[Proof of the Theorem]

From the  exact sequence  of sheaves 
$$0\to \NN_{X\slash Y}^*\to  \Omega_{Y|X}^1\to \Omega_X^1\to 0$$
we have a long   exact sequence in cohomology 

\xymatrix{H^1(X, \Omega_{Y|X}^1\otimes\End(E))\ar[r]^{\imath^*}  & H^1(X, \Omega_X^1\otimes\End(E))\ar[r] &   H^2(X,\NN_{X\slash Y}^*\otimes\End(E)).}

Under close examination  the coboundary map coincides with $\alpha\mapsto\alpha\cdot  \kappa_{X\slash Y}$, 
hence $\At(E)\cdot\kappa_{X\slash Y}=0$ iff $\At(E)$ belongs to the image of $\imath^*$. 

\end{proof} 

\section{Obstructions to uniqueness}
\label{sec:unique}

According to  \cite[Proposition 1.4]{Griffiths},   the obstructions to uniqueness of an extension of $E\to X$ over  $X_1$ are cohomology classes in $H^1(X,\NN_{X\slash Y}^*\otimes\End(E))$. (A  proof of it is also  omitted though.)  It may be  convenient to distinguish  two different  obstructions,    the vector bundles $E_1\to X_1$ extending the same  bundle $E$   can have different  Atiyah classes $\At(E_1)$, and also  there can be many bundles  with the same Atiyah  class among them. 

We assume  that the  vector bundle $E\to X$ under consideration has  at least one extension and denote $\CC(E)=\{E_1\to X_1\mid \imath^*E_1=E\}\neq\emptyset$. There is a natural action of the group   $H^1(X,\NN_{X\slash Y}^*\otimes\End(E))$ on this set which may be described as follows: for a given  $e\in \CC(E)$ with transition functions $(g_{ij}, \Delta_{ij})$  and a given cohomology class $\gamma\in H^1(X,\NN_{X\slash Y}^*\otimes\End(E))$ we denote by $e'=e+\gamma\in \CC(E)$ a bundle with transition functions $(g_{ij}, \Delta'_{ij})$ defined by  
$\Delta'_{ij}=\Delta_{ij}+g_{i\lambda} C_{ij}g_{\lambda j}$ where  $C_{ij}\in \Gamma(U_i\cap U_j, \NN_{X\slash Y}^*\otimes\End(E))$ is a cocycle representing 
$\gamma$. 
  
\begin{prop}

The map 
$$H^1(X,\NN_{X\slash Y}^*\otimes\End(E))\to \CC(E),\, \gamma\mapsto e+\gamma$$
is a bijection making  the diagram 

\xymatrix@C=2pc@R=2pc{\CC(E)\ar[r]^{\At} & H^1(X,  \Omega_{Y|X}^1\otimes\End(E))\ar[d]^{-\At(e)}\\
H^1(X,\NN_{X\slash Y}^*\otimes\End(E))\ar[r]\ar[u]^{+e} & H^1(X,  \Omega_{Y|X}^1\otimes\End(E))}
commutative. 
\end{prop}

A proof  (using the same argument as in   Proposition 1)  is   straightforward. 
As a consequence of Proposition 2 and the long exact sequence, bundles in  $\CC(E)$ with the same Atiyah class come from global sections  of the sheaf $\Omega_{X}^1\otimes\End(E)$.

\end{document}